\documentclass[10pt]{amsart}
\usepackage{amsmath}
\usepackage{amssymb}
\usepackage{amsfonts}
\usepackage{amsthm}
\usepackage{enumerate}
\usepackage[mathscr]{eucal}
\usepackage{verbatim}
\usepackage{amsthm}
\usepackage{amscd}

\usepackage{enumitem}

\makeatletter
\def\step{%
   \@ifnextchar[ \@step{\@noitemargtrue\@step[\@itemlabel]}}
\def\@step[#1]{\item[#1]\mbox{}\\\hspace*{\dimexpr-\labelwidth-\labelsep}}
\makeatother

\newtheorem{theorem}{Theorem}[section]

\newtheorem{corollary}[theorem]{Corollary}
\newtheorem{lemma}[theorem]{Lemma}

\theoremstyle{definition}
\newtheorem{definition}{Definition}
\newtheorem{remark}{Remark}

\numberwithin{equation}{section}

\newcommand\relphantom[1]{\mathrel{\phantom{#1}}}

\begin{document}

\address{School of Mathematics \\
           Korea Institute for Advanced Study, Seoul\\
            Republic of Korea}
   \email{qkrqowns@kias.re.kr}
\author{Bae Jun Park}

\title[maximal inequality]{ Some maximal inequalities   on Triebel-Lizorkin spaces for $p=\infty$}

\begin{abstract} 
In this work we give some maximal inequalities in Triebel-Lizorkin spaces, which are ``$\dot{F}_{\infty}^{s,q}$-variants" of Fefferman-Stein vector-valued maximal inequality and Peetre's maximal inequality. We will give some applications of the new maximal inequalities and discuss sharpness of some results.
\end{abstract}

\maketitle
\section{\textbf{Introduction}}

\subsection{Motivation}

Triebel-Lizorkin space $\dot{F}_{p}^{s,q}(\mathbb{R}^d)$ ( or $F_{p}^{s,q}(\mathbb{R}^d)$ )  provides a general framework that unifies the subject of classical function spaces such as $L^p$ spaces, Hardy spaces, Sobolev spaces, and $BMO$ spaces.
We recall 
\begin{align*}
& L^p\text{space} &\dot{{F}}_p^{0,2}(\mathbb{R}^d) = {F}_p^{0,2}(\mathbb{R}^d)=L^p(\mathbb{R}^d) & &1<p<\infty \\
&\text{Hardy space} &\dot{{F}}_p^{0,2}(\mathbb{R}^d)={H}^p(\mathbb{R}^d) , \quad   F_p^{0,2}(\mathbb{R}^d)=h^p(\mathbb{R}^d) & &0<p\leq 1\\
&\text{Sobolev space}&\dot{{F}}_p^{{s},2}(\mathbb{R}^d)=\dot{L}^p_{{s}}(\mathbb{R}^d), \quad  {F}_p^{{s},2}(\mathbb{R}^d)=L^p_{{s}}(\mathbb{R}^d)  & & {s} >0, 1<p<\infty\\
&BMO, bmo &\dot{{F}}_{{\infty}}^{0,2}(\mathbb{R}^d)=BMO(\mathbb{R}^d),  \quad  {{F}}_{{\infty}}^{0,2}(\mathbb{R}^d)=bmo(\mathbb{R}^d).
\end{align*}
Many results about $L^p(\mathbb{R}^d)$ and $H^p(\mathbb{R}^d)$ have been generalized to $\dot{F}_p^{s,q}(\mathbb{R}^d)$  and
 the Fefferman-Stein vector-valued maximal inequality in \cite{Fe_St0} is a key tool to develop a theory of Triebel-Lizorkin spaces. However the maximal inequality  cannot be adapted to $\dot{F}_{\infty}^{s,q}(\mathbb{R}^d)$.  The main purpose of this paper is to provide an analogous maximal inequality that can be readily used for $\dot{F}_{\infty}^{s,q}(\mathbb{R}^d)$.
 
For sake of simplicity we restrict ourselves in the sequel to function spaces defined on $\mathbb{R}^d$ and omit ``$\mathbb{R}^d$''.

\subsection{Maximal inequalities}

Let $\mathcal{M}$ be the Hardy-Littlewood maximal operator, defined by
\begin{equation*}
\mathcal{M}f(x):=\sup_{x\in Q}{\frac{1}{|Q|}\int_Q{|f(y)|}dy}
\end{equation*} where the supremum is taken over all cubes containing $x$,
 and for $0<t<\infty$ let $\mathcal{M}_tf:=\big(  \mathcal{M}(|f|^t) \big)^{1/t}$. 
Then the Fefferman-Stein vector-valued maximal inequality in \cite{Fe_St0} says  that
for $0<r<p,q<\infty$
\begin{equation}\label{hlmax}
\Big\Vert  \Big(\sum_{k}{\big(\mathcal{M}_{r}f_k\big)^q}\Big)^{1/{q}} \Big\Vert_{L^p} \lesssim  \Big\Vert \Big( \sum_{k}{|f_k|^q}  \Big)^{1/{q}}  \Big\Vert_{L^p}.
\end{equation} Here, the notation $``\lesssim"$ indicates that an unspecified constant, which may depend on $d, p, q, r$, is involved in the inequality.
Note that (\ref{hlmax}) also holds when $q=\infty$. 

Now for $k\in\mathbb{Z}$ and $\sigma>0$ we define the Peetre maximal operator $\mathfrak{M}_{\sigma,2^k}$ by \begin{equation*}
\mathfrak{M}_{\sigma,2^k}f(x):=\sup_{y\in\mathbb{R}^d}{\frac{|f(x-y)|}{(1+2^k|y|)^{\sigma}}}.
\end{equation*}
For $r>0$ let $\mathcal{E}(r)$ be the space of tempered distributions whose Fourier transforms are supported in $\{\xi:|\xi|\leq 2r\}$.
As shown in \cite{Pe} and \cite[1.3.1]{Tr} one has the majorization \begin{equation}\label{maximalbound}
\mathfrak{M}_{d/r,2^k}f(x)\lesssim_r \mathcal{M}_rf(x), 
\end{equation}  if  $f\in \mathcal{E}(2^k)$.

Then  it follows via (\ref{hlmax}) that for $0<p<\infty$, $0<q\leq\infty$, and $0< r<p,q$,
\begin{equation}\label{max}
\Big\Vert  \Big(\sum_{k}{(\mathfrak{M}_{d/r,2^k}f_k)^q}\Big)^{1/{q}} \Big\Vert_{L^p} \lesssim  \Big\Vert \Big( \sum_{k}{|f_k|^q}  \Big)^{1/{q}}  \Big\Vert_{L^p}
\end{equation} if $f_k\in \mathcal{E}(2^k)$.
It is well known that $r<p$ is necessary and
Christ and Seeger \cite{Ch_Se} proved that for $q\leq p$, the condition $r<q$ in $(\ref{max})$ is a necessary condition  by using a random construction.

It is natural to ask whether the analogous maximal inequalities for $p=\infty$ hold. Clearly, (\ref{hlmax}) and (\ref{max}) hold for $p=q=\infty$ and the example in \cite[2.1.4]{Tr2} shows that those inequalities do not hold for $p=\infty$ and $0<q<\infty$. 

Now we consider  ``$\dot{F}_{\infty}^{s,q}$-variants", which are motivated by the definition of $\dot{F}_{\infty}^{s,q}$ (or $F_{\infty}^{s,q}$).
Let  $\mathcal{D}$ denote the set of all dyadic cubes in $\mathbb{R}^d$ and $\mathcal{D}_{k}$ the subset of $\mathcal{D}$ consisting of the cubes with side length $2^{-k}$ for $k\in\mathbb{Z}$. 
  For $Q\in \mathcal{D}$, denote the side length of $Q$ by $l(Q)$, lower left corner of $Q$ by $x_Q$, and the characteristic function of $Q$ by $\chi_Q$.

Note that $\Vert f\Vert_{\dot{F}_p^{0,q}}\approx \Vert \{\Pi_kf\}\Vert_{L^{p}(l^q)}$ for $p<\infty$ or for $p=q=\infty$, but
\begin{equation*}
\Vert f\Vert_{\dot{F}_{\infty}^{0,q}}\approx\sup_{P\in\mathcal{D}}{\Big(\frac{1}{|P|}\int_P{\sum_{k=-\log_2{l(P)}}^{\infty}{|\Pi_kf(x)|^q}}dx\Big)^{1/q}},\quad q<\infty
\end{equation*}  where $\Pi_k$ is a homogeneous Littlewood-Paley frequency decomposition, defined in Section \ref{application}. Here ``$\approx$'' means both ``$\lesssim$'' and ``$\gtrsim$''.

 One may wonder whether for $0<r<q<\infty$ 
 \begin{equation}\label{yyy}
 \sup_{P\in\mathcal{D}}{\Big(  \frac{1}{|P|}\int_P{  \sum_{k=-\log_2{l(P)}}^{\infty}{  \big( \mathcal{M}_rf_k(x)  \big)^{q}      }    }dx  \Big)^{1/q}} 
 \end{equation} can be dominated by 
\begin{equation}\label{uuu}
\sup_{P\in\mathcal{D}}{\Big(  \frac{1}{|P|}\int_P{  \sum_{k=-\log_2{l(P)}}^{\infty}{   |f_k(x)|^q   }    }dx  \Big)^{1/q}}
\end{equation} provided that $f_k\in \mathcal{E}(2^k)$. However, this does not hold;
\begin{theorem}\label{counterex}
Let $0<r,q<\infty$. Then there exists a sequence  $\{f_k\}$, with $f_k\in\mathcal{E}(2^k)$, so that $(\ref{uuu})<\infty$ but
$(\ref{yyy})=\infty$.

\end{theorem}
In view of Theorem \ref{counterex} we will replace
$\mathcal{M}_r$  by an appropriate smaller maximal-type operator  so that $(\ref{yyy})$, with $\mathcal{M}_r$ replaced by the new maximal operator, is bounded by (\ref{uuu}).

\begin{definition}
For $\epsilon\geq0$, $r>0$, and $k\in\mathbb{Z}$, define 
\begin{align*}
\mathcal{M}_r^{k,\epsilon}f(x)&:=\sup_{2^kl(Q)\leq 1,x\in Q}{\Big( \frac{1}{|Q|}\int_{Q}{|f(y)|^r}dy  \Big)^{1/r}}\\
  &\relphantom{=}+\sup_{2^{k}l(Q)>1,x\in Q}{\big(2^kl(Q)\big)^{-\epsilon}\Big( \frac{1}{|Q|}\int_Q{|f(y)|^r}dy  \Big)^{1/r}}. \nonumber
\end{align*}
\end{definition}
We  observe that \begin{equation}\label{maxbound}
\mathcal{M}_r^{k,\epsilon}f(x) \lesssim \mathcal{M}_r^{k,0}f(x) \approx\mathcal{M}_rf(x).
\end{equation} 
Moreover, the pointwise estimate (\ref{maximalbound}) does not hold when we replace $\mathcal{M}_r$ by $\mathcal{M}_{r}^{k,\epsilon}$, but we will prove the following lemma, based on the idea in \cite{Pe}.
\begin{lemma}\label{lemgoal}
Let $A>0$ and suppose that $f_k\in \mathcal{E}(A2^k)$ for each $k\in \mathbb{Z}$. Then
\begin{equation*}
\mathfrak{M}_{d/r,2^k}f_k(x)\lesssim_{A,r} \mathcal{M}_t^{k,d(1/r-1/t)}f_k(x) \quad \text{for}\quad r<t,
\end{equation*} uniformly in $k$.
\end{lemma}

Now we state the maximal inequality of $\mathcal{M}_r^{k,\epsilon}$, which is  our main result.
\begin{theorem}\label{maximal1}
Let $0<r<q<\infty$ and $\epsilon>0$. Suppose that  $f_k\in\mathcal{E}(A2^k)$ for some $A>0$ and for all $k\in\mathbb{Z}$.
 Let $\mu\in\mathbb{Z}$ and $P\in\mathcal{D}_{\mu}$.
Then 
\begin{equation}
{\Big(  \frac{1}{|P|}\int_P{  \sum_{k=\mu}^{\infty}{  \big( \mathcal{M}_r^{k,\epsilon}f_k(x)  \big)^{q}      }    }dx  \Big)^{1/q}} \lesssim_{A,r} \sup_{R\in\mathcal{D}_{\mu}}{\Big(  \frac{1}{|R|}\int_R{  \sum_{k=\mu}^{\infty}{   |f_k(x)|^q   }    }dx  \Big)^{1/q}}. 
\end{equation}
Here, the implicit constant of the inequality is independent of $\mu$ and $P$.
\end{theorem}

Then as a consequence of Lemma \ref{lemgoal} and Theorem \ref{maximal1} the following result holds. 
\begin{corollary}\label{maximal2}
Let $0<r<q<\infty$. Suppose that for  $f_k\in\mathcal{E}(A2^k)$  for some $A>0$ and for all $k\in\mathbb{Z}$. Let $\mu\in\mathbb{Z}$ and $P\in\mathcal{D}_{\mu}$.
Then 
\begin{equation*}
{\Big(  \frac{1}{|P|}\int_P{  \sum_{k=\mu}^{\infty}{  \big(\mathfrak{M}_{d/r,2^k}f_k(x) \big)^{q}      }    }dx  \Big)^{1/q}} \label{boundee} \lesssim_{A,r} \sup_{R\in\mathcal{D}_{\mu}}{\Big(  \frac{1}{|R|}\int_R{  \sum_{k=\mu}^{\infty}{   |f_k(x)|^q   }    }dx  \Big)^{1/q}}.
\end{equation*}
Here, the implicit constant of the inequality is independent of $\mu$ and $P$.
\end{corollary}

\begin{remark}
 Corollary \ref{maximal2} is sharp in the sense that if $r\geq q$ then there exists a sequence $\{f_k\}$ in $\mathcal{E}(2^k)$ for which the inequality does not hold. For details see Section \ref{example}.
\end{remark}

It has been observed in \cite{Bu_Ta} and \cite{Se2} that weaker versions of maximal inequalities for $\mathfrak{M}_{\sigma,2^k}$ hold, namely that the left hand side of the asserted inequality in Corollary \ref{maximal2} is bounded by the supremum over arbitrary dyadic cubes, not over $R\in\mathcal{D}_{\mu}$, assuming $\sigma$ is large enough. That is, we provide improvements by deriving $``\sup_{R\in\mathcal{D}_{\mu}}$'' and by giving the optimal range $\sigma(=d/r)>d/q$.

As an application of Corollary \ref{maximal2}, for $\mu\in\mathbb{Z}$, $q_1<q_2<\infty$, and $\mathbf{f}:=\{f_k\}_{k\in\mathbb{Z}}$ one has
\begin{equation}\label{embinf}
\mathcal{V}_{\mu,q_2}[\mathbf{f}]\lesssim \mathcal{V}_{\mu,q_1}[\mathbf{f}],
\end{equation} provided that each $f_k$ is defined as in Corollary \ref{maximal2}, where
\begin{equation*}
\mathcal{V}_{\mu,q}[\mathbf{f}]:= \sup_{P\in\mathcal{D},l(P)\leq 2^{-\mu}}{\Big(  \frac{1}{|P|}\int_P{  \sum_{k=-\log_2{l(P)}}^{\infty}{   |f_k(x)|^q   }    }dx  \Big)^{1/q}}. \end{equation*}
Indeed,  for fixed $Q\in\mathcal{D}_{k}$ and   $\sigma>0$ 
\begin{equation}\label{qwqw}
\Vert \mathfrak{M}_{\sigma,2^k}f_k\Vert_{L^{\infty}(Q)}\lesssim_{\sigma} \inf_{y\in Q}{\mathfrak{M}_{\sigma,2^k}f_k(y)}
\end{equation} and then for $k\geq \mu$, $P\in\mathcal{D}_{\mu}$, and $\sigma>d/q_1$
\begin{align*}
  \Vert f_k\Vert_{L^{\infty}(P)}&\leq \sup_{Q\in\mathcal{D}_k, Q\subset P}{\big\Vert \mathfrak{M}_{\sigma,2^k}f_k\big\Vert_{L^{\infty}(Q)}}\\ 
  &\lesssim_{\sigma} \sup_{Q\in\mathcal{D}_k, Q\subset P}{\Big( \frac{1}{|Q|}\int_Q{\big( \mathfrak{M}_{\sigma,2^k}f_k(y)\big)^{q_1}}dy\Big)^{1/q_1}}    \lesssim \mathcal{V}_{\mu,q_1}[\mathbf{f}]
\end{align*}  where we used  Corollary \ref{maximal2} in the last inequality.
By applying $|f_k(x)|^{q_2}\lesssim \big(\mathcal{V}_{\mu,q_1}[\mathbf{f}]\big)^{q_2-q_1}|f_k(x)|^{q_1}$ for $x\in P$, one can prove (\ref{embinf}).
Furthermore, by using (\ref{qwqw}) and Corollary \ref{maximal2} one can also obtain that for $0<q<\infty$ and $\mu\in\mathbb{Z}$
\begin{equation*}
\sup_{k\geq \mu}\Vert f_k\Vert_{L^{\infty}}\lesssim \mathcal{V}_{\mu,q}[\mathbf{f}].
\end{equation*}
 Then together with (\ref{embinf}), this  implies $F_{\infty}^{s,q_1}\hookrightarrow F_{\infty}^{s,q_2}$ for all $0<q_1<q_2\leq\infty$.

This paper is organized as follows. We give some applications of the new maximal inequalities in Section \ref{application}.
We prove Lemma \ref{lemgoal} and Theorem \ref{maximal1} in Section \ref{proofmain} and construct some counter examples  in section \ref{example} to prove Theorem \ref{counterex} and to show the sharpness of Corollary \ref{maximal2}.

\section{Applications in $\dot{F}_{\infty}^{s,q}$ (or $F_{\infty}^{s,q}$)}\label{application}

Let $S$ denote the Schwartz space and $S'$ the space of tempered distributions. 
For the Fourier transform of $f$ we use the definition 
$\widehat{f}(\xi)=\int_{\mathbb{R}^d}{f(x) e^{-2\pi i\langle x,\xi \rangle}}dx $ and denote by $f^{\vee}$ the inverse Fourier transform of $f$. 
 Let $\phi$ be a smooth function so that $\widehat{\phi}$ is supported in $\{\xi:2^{-1}\leq |\xi|\leq 2\}$ and $\sum_{k\in\mathbb{Z}}{\widehat{\phi_k}(\xi)}=1$ for $\xi\not=0$ where $\phi_k:=2^{kd}\phi(2^k\cdot)$. 
 For each $k\in\mathbb{Z}$ we define convolution operators $\Pi_k$ by $\Pi_kf:=\phi_k\ast f$. Then for $0<p,q\leq \infty$ and $s\in \mathbb{R}$ the homogeneous Besov spaces $\dot{B}_p^{s,q}$ and Triebel-Lizorkin spaces $\dot{F}_p^{s,q}$  are defined as  subspaces of $S'/\mathcal{P}$ (tempered distributions modulo polynomials) with  (quasi-)norms  
 \begin{equation*}
 \Vert f\Vert_{\dot{B}_p^{s,q}}:=\big\Vert \{2^{s k}\Pi_kf\}_{k\in\mathbb{Z}}\big\Vert_{l^q(L^p)},
 \end{equation*} 
 \begin{equation*}
 \Vert f\Vert_{\dot{F}_p^{s,q}}:=\big\Vert \{2^{s k}\Pi_kf\}_{k\in\mathbb{Z}}\big\Vert_{L^p(l^q)}, \quad p<\infty ~\text{ or }~p=q=\infty 
 \end{equation*}
 respectively.
 When $p=\infty$ and $q<\infty$ we apply
\begin{equation*}
\Vert f\Vert_{\dot{F}_{\infty}^{s,q}}:=\sup_{P\in\mathcal{D}}\Big(\frac{1}{|P|}\int_P{\sum_{k=-\log_2{l(P)}}^{\infty}{2^{s kq}|\Pi_kf(x)|^q}}dx \Big)^{1/q}
\end{equation*} where $\mathcal{D}$ stands for the set of all dyadic cubes in $\mathbb{R}^d$. 

For inhomogeneous versions let $\widehat{\Phi}:=1-\sum_{k=1}^{\infty}\widehat{\phi_k}$ and define 
$\Lambda_0f:=\Phi\ast f$ and $\Lambda_kf:=\Pi_k f$ for $k\geq 1$.
Then the inhomogeneous spaces ${B}_p^{s,q}$ and $F_p^{s,q}$ are defined similarly.
For $0<p,q\leq \infty$ and $s\in \mathbb{R}$,  ${B}_p^{s,q}$ is the subspace of $S'$ with  (quasi-)norms   
\begin{equation*}
\Vert f\Vert_{{B}_p^{s,q}}:=\big\Vert \{2^{s k}\Lambda_kf\}_{k=0}^{\infty}\big\Vert_{l^q(L^p)}.
\end{equation*}
The inhomogeneous Triebel-Lizorkin spaces ${F}_{p}^{s,q}$ is a subspace of $S'$ with norms   
\begin{equation*}
\Vert f\Vert_{F_p^{s,q}}:=\big\Vert \{2^{s k}\Lambda_kf\}_{k=0}^{\infty}\big\Vert_{L^p(l^q)} \quad  p<\infty ~\text{ or }~p=q=\infty
\end{equation*} 
\begin{equation*}
\Vert f\Vert_{F_{\infty}^{s,q}}:=\Vert \Lambda_0f\Vert_{L^{\infty}}+\sup_{P\in\mathcal{D},l(P)<1}\Big(\frac{1}{|P|}\int_P{\sum_{k=-\log_2{l(P)}}^{\infty}{2^{s kq}|\Lambda_kf(x)|^q}}dx \Big)^{1/q}
\end{equation*} where  the supremum is taken over all dyadic cubes whose side length $l(P)$ is less than $1$.

\subsection{Mikhlin-H\"ormander multiplier theorems for $\dot{F}_{\infty}^{0,q}$ $(0<q<\infty)$}

For $m\in L^{\infty}$  the multiplier operator $T_m$ is defined as $T_mf(x):=\big( m\widehat{f}\big)^{\vee}(x)$.
The classical Mikhlin multiplier theorem \cite{Mik} states that if a function $m$, defined on $\mathbb{R}^d$, satisfies
\begin{equation*}
\big|  \partial_{\xi}^{\beta}m(\xi)  \big|\lesssim_{\beta}|\xi|^{-|\beta|}
\end{equation*} for all multi-indices $\beta$ with $|\beta|\leq \big[d/2\big]+1$, then the operator $T_m$ is bounded in $L^p$ for $1<p<\infty$.
In \cite{Ho} H\"ormander extends  Mikhlin's theorem to functions $m$ with the weaker condition
\begin{equation}\label{hocondition}
\mathcal{A}_{\alpha}[m]:=\sup_{k\in\mathbb{Z}}{\big\Vert m(2^k\cdot)\varphi\big\Vert_{L^2_{\alpha}}}<\infty
\end{equation} for $\alpha>d/2$ where $L^2_{\alpha}$ stands for the standard fractional Sobolev space, $\varphi$ is a cutoff function such that $0\leq \varphi\leq 1$, $\varphi=1$ on $1/2\leq |\xi|\leq 2$, and $Supp(\varphi)\subset \{1/4\leq |\xi|\leq 4\}$.
When $0<p\leq 1$ Calder\'on and Torchinsky \cite{Ca_To} proved that if (\ref{hocondition}) holds for $\alpha>d/p-d/2$, then $m$ is a Fourier multiplier of Hardy space $H^p$. A different proof was given by Taibleson and Weiss \cite{Ta_We}.

In \cite{Tr3} and \cite[p74]{Tr} Triebel extended these results to inhomogeneous Triebel-Lizorkin spaces and the arguments can also be applied to the homogeneous spaces. That is,
for $0<p,q<\infty$
if $m\in L^{\infty}$ satisfies (\ref{hocondition}) for $\alpha>d/\min{(1,p,q)}-d/2$
then
\begin{equation*}
\Vert T_m\Vert_{\dot{F}_p^{0,q}}\lesssim \mathcal{A}_{\alpha}[m]\Vert f\Vert_{\dot{F}_p^{0,q}}.
\end{equation*}

As an application of Theorem \ref{maximal1} one can extend the multiplier theorem to $\dot{F}_{\infty}^{0,q}$.
\begin{theorem}\label{oooo}
Let $0<q<\infty$. Assume $m\in L^{\infty}(\mathbb{R}^d)$ satisfies the condition (\ref{hocondition}) for $\alpha>d/\min{(1,q)}-d/2$.  Then $T_m$ is bounded in $\dot{F}_{\infty}^{0,q}$.  Moreover, in that case
\begin{equation*}
\big\Vert T_mf\big\Vert_{\dot{F}_{\infty}^{0,q}}\lesssim \mathcal{A}_{\alpha}[m] \Vert f\Vert_{\dot{F}_{\infty}^{0,q}}
\end{equation*}
\end{theorem}
For the proof the reader is referred to Theorem 2.3 and Theorem 2.4  in \cite{Park4}.

\begin{remark}
We note that the above theorem also proves the $\dot{F}_{\infty}^{s,q}$  boundedness of $T_m$ for $s\in\mathbb{R}$
because the set of all Fourier multipliers for $\dot{F}_p^{s,q}$ is independent of $s$.
This immediately implies that 
$T_m$ maps $BMO^s( = \dot{F}_{\infty}^{s,2})$ into itself for $s\in\mathbb{R}$ and $\alpha>d/2$ where the Sobolev-BMO spaces $BMO^s$ were initially introduced by Neri \cite{Ne} and further developed by Strichartz \cite{St1}. Recently, the spaces attract some attention in connection with Cauchy problems for non-linear parabolic PDEs, especially Navier-Stokes equations. See \cite{Tr13, Tr15} for details.

\end{remark}

\subsection{ ${F}_{\infty}^{s,q}$-Boundedness of Pseudo-differential operators of type $(1,1)$}

For $0<p<\infty$ and $0<q\leq\infty$ Runst \cite{Ru}, Torres \cite{To}, and Johnsen \cite{Jo} proved the boundedness of pseudo-differential operators of type $(1,1)$ on Triebel-Lizorkin spaces.
Let ${\tau}_{p,q}=\max{\{0,d({1}/{p}-1),d({1}/{q}-1)\}}$ and ${\tau}_{p}=\max{\{0,d({1}/{p}-1)\}}$. 
Suppose $s,m\in\mathbb{R}$ and $a\in \mathcal{S}_{1,1}^{m}$. Then for $0<p< \infty$ and $0<q\leq\infty$
\begin{equation*}
T_{[a]}:F_p^{s+m,q}\to F_{p}^{s,q} \quad \text{ if }~s>\tau_{p,q}.
\end{equation*}
On the other hand,  for $0<p,q\leq \infty$
\begin{equation*}
T_{[a]}:B_{p}^{s+m,q} \to B_{p}^{s,q} \quad  \text{ if } ~s>{\tau}_{p}.
\end{equation*}

As an application of Theorem \ref{maximal1} one can extend the boundedness to $F_{\infty}^{s,q}$.
\begin{theorem}
Suppose $m\in\mathbb{R}$, $0<q< \infty$, and  $a\in\mathcal{S}_{1,1}^{m}$. If $s>{\tau}_{q}$  then 
\begin{equation*}
T_{[a]}:F_{\infty}^{s+m,q}\to F_{\infty}^{s,q}.
\end{equation*}
 \end{theorem}

The proof is based on the idea in \cite{Jo, Ru}, and  Theorem \ref{maximal1} is used to derive a ``$F_{\infty}^{s,q}$-variant'' of Marschall's inequality in \cite{Ma}. We refer to \cite{Park3} which contains detailed proofs and some sharpness results.

\subsection{Franke's embedding theorem for $F_{\infty}^{s,q}$}

We recall an extension of Sobolev embedding theorem to $B_p^{s,q}$ and $F_p^{s,q}$ spaces.
Let $-\infty<s_1<s_2<\infty$ and $0<p_0<p_1\leq \infty$ with $s_0-d/p_0=s_1-d/p_1$. 
Then $F_{p_0}^{s_0,q_0}\hookrightarrow F_{p_1}^{s_1,q_1}$ for $p_1<\infty$.

This implies
$F_{p_0}^{s_0,q}\hookrightarrow F_{p_1}^{s_1,p_1}=B_{p_1}^{s_1,p_1} $ and 
$B_{p_0}^{s_0,p_0}=F_{p_0}^{s_0,p_0}\hookrightarrow F_{p_1}^{s_1,q}$ for $0<q\leq \infty$. 
Jawerth \cite{Ja} and Franke \cite{Fr} showed that these embeddings are not optimal and improved that
\begin{equation}\label{franke}
F_{p_0}^{s_0,q}\hookrightarrow B_{p_1}^{s_1,p_0} \quad \text{and} \quad B_{p_0}^{s_0,p_1}\hookrightarrow F_{p_1}^{s_1,q}, \quad \text{ for } ~p_1<\infty.
\end{equation}
They used interpolation techniques and later Vyb\'iral \cite{Vy} gave a different proof of the embeddings by using discrete characterization of $B_p^{s,q}$ and $F_p^{s,q}$. 

Now,  as an direct consquence of (\ref{embinf}), we prove the analogue of (\ref{franke}) when $p_1=\infty$. 
\begin{theorem}\label{fremb}
Suppose $0<p_0<\infty$, $0<q\leq\infty$, $s_0\in\mathbb{R}$ and $s=s_0-d/p_0$. Then
\begin{equation}
B_{p_0}^{s_0,\infty}\hookrightarrow F_{\infty}^{s,q}.
\end{equation}
\end{theorem}
\begin{proof}
We may assume $q<\infty$.
The proof is independent of the previous results, and quite simple and direct without interpolation technique and discrete characterizations.
It suffices to show \begin{equation}\label{embedpf}
\sup_{P\in\mathcal{D},l(P)<1}\Big(\frac{1}{|P|}\int_P\sum_{k=-\log_2{l(P)}}^{\infty}{2^{skq}\big| \Lambda_kf(x)\big|^q} dx\Big)^{1/q}\lesssim \Vert f\Vert_{B_{p_0}^{s_0,\infty}}
\end{equation}
If $q<p_0$, then H\"older's inequality yields that the left hand side of (\ref{embedpf}) is less than
\begin{equation*}
\sup_{P\in\mathcal{D},l(P)<1}\Big( \frac{1}{|P|^{q/p_0}}\sum_{k=-\log_2{l(P)}}^{\infty}{2^{-kdq/p_0}\big\Vert 2^{s_0k}\big| \Lambda_kf\big|\big\Vert_{L^{p_0}}^{q}}\Big)^{1/q}
\end{equation*} and this is clearly dominated by $\Vert f\Vert_{B_{p_0}^{s_0,\infty}}$.
If $q\geq p_0$ then (\ref{embinf}) proves that the left hand side of (\ref{embedpf}) is bounded by a constant times
\begin{equation*}
\sup_{P\in\mathcal{D},l(P)<1}\Big( \frac{1}{|P|}\sum_{k=-\log_2{l(P)}}^{\infty}{2^{-kd}\big\Vert 2^{s_0k}\big| \Lambda_kf\big|\big\Vert_{L^{p_0}}^{p_0}}\Big)^{1/p_0},
\end{equation*} which is less than $\Vert f\Vert_{B_{p_0}^{s_0,\infty}}$ by the same reason.
\end{proof}
Note that Theorem \ref{fremb} immediately proves that the above Sobolev-type embedding $F_{p_0}^{s_0,q_0}\hookrightarrow F_{p_1}^{s_1,q_1}$ also holds for $p_1=\infty$.

\section{\textbf{Proof of Lemma \ref{lemgoal} and Theorem \ref{maximal1} }}\label{proofmain}

\subsection{{Proof of Lemma \ref{lemgoal}}}
We follow arguments in the proof of Lemma 2.1 in \cite{Pe}. Let $0<r<t<\infty$ and $f_k\in\mathcal{E}(A2^k)$. By translation-invariance it suffices to show 
\begin{equation}\label{reduction0}
\mathfrak{M}_{d/r,2^k}f_k(0)\lesssim_{A,r}\mathcal{M}_t^{k,d(1/r-1/t)}f_k(0).
\end{equation}
Set $\epsilon:=d(1/r-1/t)>0$ and let $0<\delta<1$.
By the mean value theorem we obtain that
\begin{equation*}
|f_k(z)| \leq 2^{-k}\delta \sup_{|y|<2^{-k}\delta}{\big| \nabla f_k(z-y)  \big|}+\Big( \frac{1}{\big( 2^{-k}\delta\big)^d}\int_{|y|<2^{-k}\delta}{|f_k(z-y)|^t}dy  \Big)^{1/t}.
\end{equation*}
We see that
\begin{align*}
\Big( \frac{1}{\big( 2^{-k}\delta\big)^d}\int_{|y|<2^{-k}\delta}{|f_k(z-y)|^t}dy  \Big)^{1/t}&\leq \Big(  \frac{1}{(2^{-k}\delta)^d}\int_{|y|<|z|+2^{-k}}{|f_k(y)|^t}dy    \Big)^{1/t}\\
&\leq \frac{1}{\delta^{d/t}}(1+2^k|z|)^{\epsilon+d/t}\mathcal{M}_t^{k,\epsilon}f_k(0)
\end{align*}
Since $\epsilon+d/t=d/r$ we have
\begin{equation*}
\frac{|f_k(z)|}{(1+2^k|z|)^{d/r}}\leq 2^{-k}\delta \sup_{|y|<2^{-k}\delta}{\frac{|\nabla f_k(z-y)|}{(1+2^k|z|)^{d/r}}}+\frac{1}{\delta^{d/t}}\mathcal{M}_t^{k,\epsilon}f_k(0).
\end{equation*}
It was proved in \cite{Pe} that there exists $C_{A,r}>0$ ( $C_{A,r}=C 2^{d/r}\max{(2A,\frac{1}{2A})}A$ for some $C>0$ ) such that
\begin{equation*}
\sup_{|y|<2^{-k}\delta}{\frac{|\nabla f_k(z-y)|}{(1+2^k|z|)^{d/r}}}\leq C_{A,r}2^k  \sup_{\widetilde{z}}{\frac{|f_k(\widetilde{z})|}{(1+2^k|\widetilde{z}|)^{d/r}}},
\end{equation*}
 and this yields that
\begin{equation*}
(1-C_{A,r}\delta)\mathfrak{M}_{d/r,2^k}f_k(0)\leq \delta^{-d/t}\mathcal{M}_t^{k,\epsilon}f_k(0).
\end{equation*}
By choosing $0<\delta<1$ sufficiently small we can get (\ref{reduction0}).

\subsection{{Proof of Theorem \ref{maximal1}}}
In the proof the constant $A$ plays a minor role and affects the result only up to a constant. Thus we now assume $A=2^{-2}$.
This assumption may change slightly the definition of $\mathcal{M}_r^{k,\epsilon}f(x)$, but it is still acceptable because for a fixed constant $C>0$
\begin{equation*}
\sup_{C2^kl(Q)\leq 1, x\in Q}{\Big(\big(|f|^r\big)_Q\Big)^{1/r}}+\sup_{C2^kl(Q)>1, x\in Q}{(2^kl(Q))^{-\epsilon}\Big(\big(|f|^r\big)_Q\Big)^{1/r}}\lesssim_C\mathcal{M}_r^{k,\epsilon}f(x)
\end{equation*} where $\big(|f|^r\big)_Q:=\frac{1}{|Q|}\int_{Q}|f(y)|dy$.

Suppose $f_k\in\mathcal{E}(2^{k-2})$ and
let $\psi_k$ be a Schwartz function whose Fourier transform takes the value $1$ on the support of $\widehat{f_k}$ and is supported in $\{\xi: |\xi|\leq 2^{{k}}\}$. Then by using the idea in the proof of \cite[Lemma 2.1]{Fr_Ja0} we write
\begin{equation}\label{pointdecompose}
f_k(x)    =\sum_{Q\in\mathcal{D}_{{k}}}{2^{-{k}d}f_k(x_Q)\psi_k(x-x_Q)}.
\end{equation} 

We first claim that for a dyadic cube $P$ with $l(P)\geq 2^{-{k}}$ 
\begin{equation}\label{sub}
\Big(  \sum_{Q\in\mathcal{D}_{{k}}(P)}{|f_k(x_Q)|^q}   \Big)^{1/q}\lesssim \big( 2^{{k}}l(P)  \big)^{d/q}\sup_{R\in\mathcal{D},l(R)=l(P)}{\Big( \frac{1}{|R|}\int_R{|f_k(x)|^q}  dx \Big)^{1/q}}.
\end{equation}
Let  $\sigma>d/q$ and $\gamma$ be a Schwartz function so that $\gamma(0)=1$, $Supp(\widehat{\gamma})\subset \{\xi:|\xi|\leq 1\}$, and $\gamma_k(x)=\gamma(2^{{k}}x)$. For each $Q\in\mathcal{D}_{{k}}$ define
${g_{k,Q}}(x)=f_k(x)\gamma_k(x-x_Q)$.
Observe that ${g_{k,Q}}(x_Q)=f_k(x_Q)$, the Fourier transform of ${g_{k,Q}}$ is supported in $\{\xi:|\xi|\leq 2^{{k+2}}\}$, and for all $y\in Q$  and arbitrary $\sigma>0$
\begin{equation}\label{observe}
|f_k(x_Q)|=|{g_{k,Q}}(x_Q)|\lesssim \mathfrak{M}_{\sigma,2^k}{g_{k,Q}}(y)
\end{equation} uniformly in $k$ and $Q$ due to (\ref{qwqw}).
Thus we have
\begin{equation*}
\sum_{Q\in\mathcal{D}_{{k}}(P)}{|f_{k}(x_Q)|^q\chi_Q(y)}  \lesssim \sum_{Q\in\mathcal{D}_{{k}}(P)}{\big(\mathfrak{M}_{\sigma,2^k}g_{k,Q}(y)\big)^q\chi_Q(y)}
\end{equation*}
By taking an integral in $y$ variable and using the $L^q$ boundedness of $\mathfrak{M}_{\sigma,2^k}$ we obtain
\begin{equation*}
2^{-{k}d}\sum_{Q\in\mathcal{D}_{{k}}(P)}{\big| f_k(x_Q) \big|^q}\lesssim \sum_{Q\in\mathcal{D}_{{k}}(P)}{\Vert  g_{k,Q} \Vert_{L^q}^q}.
\end{equation*}
Furthermore, for $Q\in\mathcal{D}_k(P)$
\begin{align*}
  \Vert  g_{k,Q} \Vert_{L^q}^q  &= {  \int_{\mathbb{R}^d}{|f_k(x)|^q\big|  \gamma\big(2^{{k}}(x-x_Q)\big) \big|^q}dx   }\\
  &\lesssim_M { \sum_{m\in\mathbb{Z}^d}{\int_{P+l(P)m}{|f_k(x)|^q\frac{1}{\big( 1+2^{{k}}|x-x_Q| \big)^{2Mq}}}dx}    }\\
    &\lesssim \sum_{m\in\mathbb{Z}^d}{\frac{1}{\big( 1+|m| \big)^{Mq}}\int_{P+l(P)m}{  |f_k(x)|^q\frac{1}{\big( 1+2^{{k}}|x-x_Q| \big)^{Mq}}   }dx  }
\end{align*} 
for sufficiently large $M$. By putting together and taking the supremum of the integral over $m\in\mathbb{Z}^d$ we derive (\ref{sub}). Here, we used the fact that $ \sum_{Q\in\mathcal{D}_k}{\frac{1}{\big( 1+2^k|x-x_Q|\big)^{Mq}}}\lesssim 1$ for $Mq>d$.

Now fix $\mu\in\mathbb{Z}$ and  $P\in\mathcal{D}_{\mu}$, and then consider
\begin{equation}\label{leftside}
{\Big(  \frac{1}{|P|}\int_P{  \sum_{k=\mu}^{\infty}{  \big( \mathcal{M}_r^{k,\epsilon}f_k(x)  \big)^{q}      }    }dx  \Big)^{1/q}}. 
\end{equation}

For each $n\in\mathbb{Z}^d$ and $P\in\mathcal{D}$  let $P+l(P)n:=\{x+l(P)n:x\in P\}$ and define  $\mathcal{D}_k(P,n)$ to be the subfamily of $\mathcal{D}_k$ that contains any dyadic cubes contained in $P+l(P)n$. When $n=0\in\mathbb{Z}^d$ let $\mathcal{D}_k(P):=\mathcal{D}_{k}(P,0)$.
We decompose (\ref{leftside}) by using
\begin{align*}
\mathcal{M}_r^{k,\epsilon}f_k  &\lesssim \mathcal{M}_r\Big(  \sum_{|m|\leq 4d}{\sum_{Q\in\mathcal{D}_k(P,m)}{2^{-kd}f_k(x_Q)\psi_k(\cdot-x_Q)     }}   \Big)\\
  &\relphantom{=}+\mathcal{M}_r^{k,\epsilon}\Big(  \sum_{|m|> 4d}{\sum_{Q\in\mathcal{D}_k(P,m)}{2^{-kd}f_k(x_Q)\psi_k(\cdot-x_Q)     }}   \Big)\\
  &=:A_k(x)+B_k(x)
\end{align*} where (\ref{pointdecompose}) and (\ref{maxbound}) are applied.
By the $L^q$ boundedness of $\mathcal{M}_r$ 
\begin{align}
&\Big(\frac{1}{|P|}\int_P{\sum_{k=\mu}^{\infty}{\big|A_k(x)\big|^q}}dx\Big)^{1/q}\nonumber\\
&\lesssim \Big(  \frac{1}{|P|}{\sum_{k=\mu}^{\infty}{     \Big\Vert    \sum_{|m|\leq 4d}{\sum_{Q\in\mathcal{D}_k(P,m)}{2^{-kd}f_k(x_Q)\psi_k(\cdot-x_Q)     }}    \Big\Vert_{L^q}^{q}   }}    \Big)^{1/q}\nonumber\\
&\lesssim \sum_{|m|\leq 4d}\Big(  \frac{1}{|P|}{\sum_{k=\mu}^{\infty}{     \Big\Vert   {\sum_{Q\in\mathcal{D}_k(P,m)}{2^{-kd}f_k(x_Q)\psi_k(\cdot-x_Q)     }}    \Big\Vert_{L^q}^{q}   }}    \Big)^{1/q}.\label{fiest}
\end{align}
Let \begin{equation*}
E_{P,m}^0(x):=\big\{  Q\in\mathcal{D}_{k}(P,m):  |x-x_Q|<2^{-{k}}  \big\}.
\end{equation*}
and for each $l\in \mathbb{N}$
\begin{equation*}
E_{P,m}^l(x):=\big\{  Q\in\mathcal{D}_{k}(P,m): 2^{-{k}}2^{l-1}\leq |x-x_Q|<2^{-{k}}2^{l}  \big\}.
\end{equation*}
Then for $0<s<\min{(1,r)}$
\begin{align*}
 & \sum_{Q\in\mathcal{D}_{k}(P,m)}{2^{-{k}d}|f_k(x_Q)|\big|\psi_k(x-x_Q)\big|}    
 \lesssim_M\sum_{l=0}^{\infty}{2^{-lM}\Big(\sum_{Q\in E_{P,m}^l(x)}{  |f_k(x_Q)|^s }\Big)^{1/s}}\\
&= \sum_{l=0}^{\infty}{{2^{-l(M-d/s)}}\Big(\frac{1}{2^{(l-k)d}}\int_{\mathbb{R}^d}\sum_{Q\in E_{P,m}^l(x)}{  |f_k(x_Q)|^s\chi_Q(y) }dy\Big)^{1/s}}\\
&\lesssim \mathcal{M}_s\Big( \sum_{Q\in\mathcal{D}_{k}(P,m)}{|f_k(x_Q)|\chi_Q}  \Big)(x)
\end{align*} for $M>d/s$.
By applying the $L^q$ boundedness of $\mathcal{M}_s$, (\ref{observe}), and the argument we did for (\ref{sub}), we obtain
\begin{equation*}
\Big\Vert    \sum_{Q\in\mathcal{D}_{k}(P,m)}{2^{-{k}d}|f_k(x_Q)|\big|\psi_k(\cdot-x_Q)\big|}    \Big\Vert_{L^q}^q \lesssim_N \sum_{n\in\mathbb{Z}^d}{\frac{1}{\big( 1+|n| \big)^{Nq}}\int_{P+l(P)m+l(P)n}{{|f_k(x)|^q}}dx}
\end{equation*} for sufficiently large $N$, which proves
\begin{align*}
(\ref{fiest})&\lesssim \sum_{|m|\leq 4d}\Big(\sum_{n\in\mathbb{Z}^d}{ \frac{1}{(1+|n|)^{Nq}}{\frac{1}{|P|} \int_{P+l(P)m+l(P)n}{\sum_{k=\mu}^{\infty}{|f_k(x)|^q}}dx   }  }    \Big)^{1/q}\\
 &\lesssim  \sup_{R\in\mathcal{D}_{\mu}}{\Big(  \frac{1}{|R|}\int_R{  \sum_{k=\mu}^{\infty}{   |f_k(x)|^q   }    }dx  \Big)^{1/q}}. 
\end{align*}

Now it remains to show
\begin{equation}\label{remaingoal}
 \Big[  \frac{1}{|P|}\int_P{\sum_{k=\mu}^{\infty}{     \big|B_k(x) \big|^q}}dx    \Big]^{1/q}\lesssim \sup_{R\in\mathcal{D}_{\mu}}{\Big(  \frac{1}{|R|}\int_R{  \sum_{k=\mu}^{\infty}{   |f_k(x)|^q   }    }dx  \Big)^{1/q}}. 
\end{equation}
The left hand side of the inequality is bounded by
\begin{equation*}
\Big(  \frac{1}{|P|}\int_P{ \sum_{k=\mu}^{\infty}{   \big(\mathcal{I}_P^{k,r}[f_k](x) \big)^q     }      dx}   \Big)^{1/q}
+\Big(  \frac{1}{|P|}\int_P{ \sum_{k=\mu}^{\infty}{  2^{-k\epsilon q} \big(\mathcal{J}_P^{k,r,\epsilon}[f_k](x) \big)^q     }      dx}   \Big)^{1/q}
\end{equation*} 
where 
\begin{equation*}
\mathcal{I}_P^{k,r}[f_k](x):=\sup_{V:x\in V,2^kl(V)\leq 1}{\Big(\frac{1}{|V|}\int_V{\Big|\sum_{|m|>4d}{\sum_{Q\in\mathcal{D}_k(P,m)}{2^{-kd}f_k(x_Q)\psi_k(y-x_Q)}}    \Big|^r}dy \Big)^{1/r}}
\end{equation*}
\begin{equation*}
\mathcal{J}_P^{k,r,\epsilon}[f_k](x):=\sup_{V:x\in V,2^kl(V)> 1}{\Big(\frac{1}{|V|^{1+\epsilon r/d}}\int_V{\Big|\sum_{|m|> 4d}{\sum_{Q\in\mathcal{D}_k(P,m)}{2^{-kd}f_k(x_Q)\psi_k(y-x_Q)}}    \Big|^r}dy \Big)^{1/r}}.
\end{equation*}

We observe that if $l(V)\leq 2^{-k}$, $x\in V\cap P$, $x_Q\in P+l(P)m$, $y\in V$, and $|m|>4d$, then $|y-x_Q|\geq |x-x_Q|-|x-y|\gtrsim |m|l(P)$ and thus
\begin{equation*}
| \psi_k(y-x_Q)|\lesssim_L\frac{2^{kd}}{(2^kl(P)|m|)^{2L}}
\end{equation*} for any $L>0$.
Choose $L>d/q$ and then 
\begin{align}\label{method2}
\Big(  \frac{1}{|P|}\int_P{ \sum_{k=\mu}^{\infty}{   \big(\mathcal{I}_P^{k,r}[f_k](x) \big)^q     }      dx}   \Big)^{1/q}&\lesssim \Big(  \sum_{k=\mu}^{\infty}{  \Big( \sum_{|m|>4d}{\sum_{Q\in\mathcal{D}_k(P,m)}{|f_k(x_Q)|\frac{1}{\big( 2^kl(P)|m|  \big)^{2L}}}}     \Big)^{q}     }  \Big)^{1/q}\nonumber\\
&\lesssim_L \Big(    \sum_{k=\mu}^{\infty}{ \sum_{|m|>4d}{\frac{1}{|m|^{Lq}}\sum_{Q\in\mathcal{D}_k(P,m)}{|f_k(x_Q)|^q\frac{1}{\big(2^kl(P)\big)^{Lq}}}    }     }  \Big)^{1/q}\nonumber\\
&\lesssim \sup_{R\in\mathcal{D}_{\mu}}{\Big(  \frac{1}{|R|}\int_R{  \sum_{j=\mu}^{\infty}{   |f_j(x)|^q   }    }dx  \Big)^{1/q}}  
\end{align} where the second inequality follows from H\"older's inequality if $q>1$, $l^q\hookrightarrow l^1$ if $q\geq 1$, and the third one from (\ref{sub}) for sufficiently large $L>0$.

For the remaining term (which is corresponding to $\mathcal{J}_P^{k,r,\epsilon}[f_k]$), if $|m|\geq 10|x-y|/l(P)$, $x\in P\cap V$, and $x_Q\in P+l(P)|m|$ then
$|y-x_Q|\geq |x-x_Q|-|y-x|\gtrsim l(P)|m|$.
Therefore for $M>d+d/q$
\begin{align*}
& \sup_{V:x\in V, 2^kl(V)>1}\Big(\frac{1}{|V|^{1+\epsilon r/d}}\int_V{\Big|\sum_{|m|>4d, |m|\geq 10|x-y|/l(P)}{\sum_{Q\in\mathcal{D}_k(P,m)}{2^{-kd}f_k(x_Q)\psi_k(y-x_Q)}} \Big|^r}dy \Big)^{1/r}\\
&\lesssim_L 2^{k\epsilon}\sum_{|m|>4d}\sum_{Q\in\mathcal{D}_k(P,m)}{|f_k(x_Q)|\frac{1}{(2^kl(P)|m|)^{2L}}}
\end{align*} for any $L>0$.
Then we obtain the bound (\ref{method2}) for  $|m|\geq 10|x-y|/l(P)$  by the same way.

Similary, if $|m|\leq 10^{-1}|x-y|/l(P)$, then we have $|y-x_Q|\geq |y-x|-|x-x_Q|\gtrsim l(P)|m|$ for $x\in P\cap V$, and $x_Q\in P+l(P)|m|$.  By applying the same technique above we obtain the desired result.

Now consider the case $10^{-1} |x-y|/l(P)<|m|<10|x-y|/l(P)$ in $\mathcal{J}_P^{k,r,\epsilon}[f_k]$. 
In this case we observe that for $x,y\in V$
\begin{equation}\label{mdmd1}
|m|l(P)\lesssim l(V).
\end{equation} 
Let $\mathcal{B}_{l(P)}(x,y):=\{m\in\mathbb{Z}^d:|m|>4d, 10^{-1}|x-y|/l(P)<|m|<10|x-y|/l(P)\}$.
Then for $L>0$ sufficiently large
\begin{align*}
& \frac{1}{|V|^{1+\epsilon r/d}}\int_V{\Big|\sum_{m\in\mathcal{B}_{l(P)}(x,y)}{\sum_{Q\in\mathcal{D}_k(P,m)}{2^{-kd}f_k(x_Q)\psi_k(y-x_Q)}} \Big|^r}dy\\
&\lesssim_L \frac{1}{|V|^{1+\epsilon r/d}}\int_V{\Big(\sum_{m\in\mathcal{B}_{l(P)}(x,y)}{\sum_{Q\in\mathcal{D}_k(P,m)}{|f_k(x_Q)|\frac{1}{(1+2^k|y-x_Q|)^{2L}}}} \Big)^r}dy\\
&\lesssim \frac{1}{|V|^{1+\epsilon r/d}}\int_V{\sum_{m\in\mathcal{B}_{l(P)}(x,y)}{\sum_{Q\in\mathcal{D}_k(P,m)}{|f_k(x_Q)|^r\frac{1}{(1+2^k|y-x_Q|)^{Lr}}}} }dy\\
&\lesssim \int_V{\sum_{m\in\mathcal{B}_{l(P)}(x,y)}{\sum_{Q\in\mathcal{D}_k(P,m)}{\frac{1}{(|m|l(P))^{d+\epsilon r}}|f_k(x_Q)|^r\frac{1}{(1+2^k|y-x_Q|)^{Lr}}}} }dy\\
&\lesssim \frac{1}{l(P)^{\epsilon r}}\sum_{|m|>4d}{\frac{1}{|m|^{d+\epsilon r}}\frac{1}{(2^kl(P))^d}\sum_{Q\in\mathcal{D}_k(P,m)}{|f_k(x_Q)|^r}}
\end{align*}
where the second one follows from H\"older's inequality if $r>1$, $l^r\hookrightarrow l^1$ if $r\leq 1$, and the third one from (\ref{mdmd1}).
Now by using H\"older's inequality with $q/r>1$ and (\ref{sub}) the last expression is less than
\begin{equation*}
2^{\mu\epsilon r}\sup_{R\in\mathcal{D}_{\mu}}\Big(\frac{1}{|R|}\int_R{\sum_{j=\mu}^{\infty}{|f_j(x)|^q}}dx \Big)^{r/q}
\end{equation*} and finally, this completes the estimate of the term corresponding to $\mathcal{J}_P^{k,r,\epsilon}[f_k]$.
That is,
\begin{equation*}
\Big(  \frac{1}{|P|}\int_P{ \sum_{k=\mu}^{\infty}{  2^{-k\epsilon q} \big(\mathcal{J}_P^{k,r,\epsilon}[f_k](x) \big)^q     }      dx}   \Big)^{1/q}\lesssim \sup_{R\in\mathcal{D}_{\mu}}\Big(\frac{1}{|R|}\int_R{\sum_{j=\mu}^{\infty}{|f_j(x)|^q}}dx \Big)^{1/q}.
\end{equation*}
Combining this with (\ref{method2}) we obtain (\ref{remaingoal}).

\section{\textbf{Some examples}}\label{example}

In what follows let $\eta$, $\widetilde{\eta}$ denote Schwartz functions so that $\eta \geq 0$, $\eta(x)\geq c$ on $\{x:|x|\leq 1/100\}$ for some $c>0$, $Supp(\widehat{\eta})\subset \{\xi: |\xi|\leq 1/100\}$, $\widehat{\widetilde{\eta}}(\xi)=1$ for $|\xi|\leq 1/100$, and $Supp(\widehat{\widetilde{\eta}})\subset \{\xi: |\xi|\leq 1/10\}$.

\subsection{Proof of Theorem \ref{counterex}}
We construct a sequence of functions $f_k\in\mathcal{E}(2^k)$ satisfying $(\ref{uuu})<\infty$, but $(\ref{yyy})=\infty$.
One key idea is that for arbitrary $\alpha>0$  there exists $M=M(\alpha)$ such that
\begin{equation}\label{boundd}
\sum_{k=1}^{\infty}{\dfrac{1}{\big( 1+|2^{-k}x-k^{\alpha}|\big)^M}}
\end{equation} is a bounded function in $\mathbb{R}$. 
Indeed, if $|x|<4\cdot 2^{\alpha}$ then $1+|2^{-k}x-k^{\alpha}|\gtrsim k^{\alpha}$ for $k\geq 3$ and thus
\begin{equation*}
(\ref{boundd})\lesssim 2+\sum_{k=3}^{\infty}{{k^{-\alpha M}}}<\infty, \quad \text{for } M>1/\alpha.
\end{equation*}
If $2^{n}n^{\alpha}\leq |x|<2^{n+1}(n+1)^{\alpha}$ for $n\geq 2$ then we have
$1+|2^{-k}x-k^{\alpha}|\gtrsim k^{\alpha}$ for $k\leq n-1$ or $k\geq n+2$. This yields that for $M>1/\alpha$
\begin{equation*}
(\ref{boundd})\lesssim 2+\sum_{k=1}^{n-1}{k^{-\alpha M}}+\sum_{k=n+2}^{\infty}{k^{-\alpha M}}<\infty \quad \text{uniformly in } n\geq 2.
\end{equation*}

On the other hand, for any $M> 1$,
\begin{equation*}
\sum_{k=1}^{\infty}{\dfrac{1}{\big( 1+|x-k^{\alpha}|\big)^M}}
\end{equation*}
 is bounded only when $\alpha\geq 1$.

For $\alpha>0$ to be chosen later, define 
\begin{equation*}
f_k(x):={\eta}\big(2^{-k}(x-2^kk^{\alpha}e_1)  \big)e^{2\pi i\langle x,2^ke_1\rangle} \quad \text{for}\quad k\geq 1
\end{equation*} and set $f_k:=0$ for $k<1$ where $e_1=(1,0,\dots,0)\in \mathbb{R}^d$.
Then we observe that for $P\in\mathcal{D}$ and sufficiently large $M>0$
\begin{equation*}
\Big( \frac{1}{|P|}\int_P{\sum_{k=-\log_2{l(P)}}^{\infty}{\big|f_k(x)\big|^q}}dx \Big)^{1/q} \lesssim \Big( \frac{1}{|P|}\int_P{\sum_{k=1}^{\infty}{\dfrac{1}{\big(1+\big|2^{-k}x-k^{\alpha}e_1\big|\big)^{Mq}}}}dx \Big)^{1/q}\lesssim_{\alpha}1
\end{equation*} uniformly in $P$.

On the other hand, 
\begin{equation*}
\sup_{P\in\mathcal{D}}\Big( \frac{1}{|P|}\int_P{\sum_{k=-\log_2{l(P)}}^{\infty}{\big|\mathcal{M}_rf_k(x)\big|^q}}dx \Big)^{1/q}\gtrsim  \Big( \int_{[0,1/2]^d}{\sum_{k=1}^{\infty}{\big|\mathcal{M}_rf_k(x)\big|^q}}dx \Big)^{1/q}
\end{equation*}
and for $x \in [0,1/2]^d$
\begin{align*}
\mathcal{M}_rf_k(x)&\gtrsim \Big( \dfrac{1}{2^{kd}k^{\alpha d}}\int_{|x-y|\lesssim 2^kk^{\alpha}}{       \big| {\eta}\big(2^{-k}(y-2^kk^{\alpha}e_1)  \big)   \big|^r}dy   \Big)^{1/r}\\
    &\gtrsim \Big( \dfrac{1}{2^{kd}k^{\alpha d}}\int_{|y-2^{k}k^{\alpha}e_1|\lesssim 2^k}{       \big| {\eta}\big(2^{-k}(y-2^kk^{\alpha}e_1)  \big)   \big|^r}dy   \Big)^{1/r}\gtrsim k^{-\alpha d/r}.
\end{align*}
By choosing $\alpha=\frac{r}{qd}$ we obtain
\begin{equation*}
\sup_{P:l(P)<1}\Big( \frac{1}{|P|}\int_P{\sum_{k=-\log_2{l(P)}}^{\infty}{\big|\mathcal{M}_rf_k(x)\big|^q}}dx \Big)^{1/q}\gtrsim \Big(  \sum_{k=1}^{\infty}{\frac{1}{k}} \Big)^{1/r}  =\infty.
\end{equation*}

\subsection{{Sharpness of Corollary \ref{maximal2}}}
We apply the idea in \cite[Section 5]{Ch_Se} to prove that the condition $r<q$ is necessary in Corollary \ref{maximal2}.
Fix $\mu\in \mathbb{Z}$.
Let $M(=M_d)>0$ be sufficiently large (so that $\sqrt{d}\ll 2^M$) and let $\beta$ be a nonnegative Schwartz function such that $\beta(x)\geq 1$ on $[-2^{-M},2^{-M}]^d$ and $Supp(\widehat{\beta})\subset \{\xi:  |\xi|\leq 1\}$.
Let $\beta_k:=2^{kd}\beta(2^k\cdot)$.
We fix $N(=N_{M,d,\mu})$ sufficiently large and for $k\geq N+\mu$ let $\mathcal{Z}_{k,N}^{d,\mu}:=\{0,1,\dots,2^{k-N-\mu}-1\}^d$. For each $j=(j_1,\dots,j_d)\in \mathcal{Z}_{k,N}^{d,\mu}$ we set
\begin{equation*}
Q_{k,j}:=[j_12^{-k+N},j_12^{-k+N}+2^{-k-M}]\times\dots\times [j_d2^{-k+N},j_d2^{-k+N}+2^{-k-M}] .   
\end{equation*}
Denote by $\chi_{k,j}$ the characteristic function of $Q_{k,j}$ and let $ h_k^{\mu}:=\sum_{j\in\mathcal{Z}_{k,N}^{d,\mu}}{\chi_{k,j}}$.
Then define $f_k^{\mu}:=h_k^{\mu}\ast \beta_k$ so that $Supp(\widehat{f_{k}^{\mu}})\subset \{\xi:  |\xi|\leq 2^k\}$.

Our claim is that
\begin{equation}\label{upper1}
\sup_{P\in\mathcal{D}_{\mu}}{\Big(\frac{1}{|P|}\int_P \sum_{k=N+\mu}^{2^{Nd}}{|f_k^{\mu}(x)|^q} dx   \Big)^{1/q}}\lesssim 1
\end{equation} uniformly in $N$,
and 
\begin{equation}\label{lower1}
\sup_{P\in\mathcal{D}_{\mu}}{\Big(\frac{1}{|P|}\int_P \sum_{k=N+\mu}^{2^{Nd}}{\big(\mathfrak{M}_{\sigma,2^k}f_k^{\mu}(x)\big)^q} dx   \Big)^{1/q}}\gtrsim_M \max{\{2^{N(d/q-\sigma)},N^{1/q}   \}}
\end{equation} for $\sigma\leq d/q$.

We first observe the pointwise estimate $\big|\beta_k\ast h_k^{\mu}(x)\big|\lesssim \mathcal{M}_sh_k^{\mu}(x)$ for $0<s<q$, which is proved in \cite{Ch_Se}. 
Then the $L^q$ boundedness of $\mathcal{M}_s$ yields that the left hand side of (\ref{upper1}) is less than a constant times
\begin{equation*}
{\Big( 2^{\mu d}\sum_{k=N+\mu}^{2^{Nd}}{\int_{\mathbb{R}^d}{ h_k^{\mu}(x)  }dx}\Big)^{1/q}}
\end{equation*} and
a straightforward computation gives that this is bounded above uniformly in  $N$.

To establish $(\ref{lower1})$ we see that
$f_k^{\mu}\geq \beta_k\ast \chi_{k,j}$ for each $j$ since $\beta_k\geq 0$,  and this proves
\begin{equation*}
\mathfrak{M}_{\sigma,2^k}f_k^{\mu}(x)\geq \mathfrak{M}_{\sigma,2^k}\big(\beta_k\ast\chi_{k,j}\big)(x).
\end{equation*}
Moreover, let $c_{k,j}$ be the center of $Q_{k,j}$ and observe that $\beta_k\ast h_{k,j}^{\mu}(c_{k,j})\geq 2^{-Md}$. Here, we used the lower bound of $\beta$ on $[-2^M,2^M]^d$.
Then as in \cite{Ch_Se}, we can prove that for $|x-c_{k,j}|\leq 2^{-k+N-2}$
\begin{equation*}
\mathfrak{M}_{\sigma,2^k}f_k^{\mu}(x)\geq \mathfrak{M}_{\sigma,2^k}(\beta_k\ast \chi_{k,j})(x)\geq 2^{-Md}\frac{1}{\big( 1+2^k|x-c_{k,j}|\big)^{\sigma}}.
\end{equation*} 
By using this estimate we obtain that the left hand side of (\ref{lower1}) is bounded below by
\begin{align*}
& \Big(2^{\mu d}\sum_{k=N+\mu}^{2^{Nd}}{ \int_{[0,2^{-\mu}]^d}{\big( \mathfrak{M}_{\sigma,2^k}f_k^{\mu}(x)\big)^q}dx   } \Big)^{1/q}\nonumber\\
&\gtrsim_M \Big(2^{\mu d}\sum_{k=N+\mu}^{2^{Nd}}{\sum_{j\in\mathcal{Z}_{k,N}^{d,\mu}}{\int_{ |x-c_{k,j}|\leq 2^{-k+N-2}}{\frac{1}{\big(1+2^k|x-c_{k,j}| \big)^{\sigma q}}}dx  }} \Big)^{1/q}
\end{align*}
We see that the last expression is $\gtrsim 2^{N(d/q-\sigma)}$ if $\sigma<d/q$ and $\gtrsim N^{1/q}$ if $\sigma=d/q$. Thus, by taking $N$ sufficiently large we can prove the sharpness.

\section*{Acknowledgement}

{
The author would like to thank his advisor Andreas Seeger for the guidance and helpful discussions.
The author also thanks Professor Hans Triebel for his helpful suggestions and comments on the subjects.
The author thanks the referees for carefully reading the manuscript.
The author was supported in part by NSF grant DMS 1500162}.

\end{document}